\documentclass{birkmult}

\usepackage{amsmath}
\usepackage{amsfonts}
\usepackage{amsbsy}
\usepackage{amssymb}
\usepackage{mathrsfs}
\usepackage{exscale}
\usepackage[colorlinks,
            bookmarksopen,
            bookmarksnumbered
           ]{hyperref}

\newtheorem{theorem}{Theorem}[section]
\newtheorem{lemma}[theorem]{Lemma}
\newtheorem{proposition}[theorem]{Proposition}

\theoremstyle{definition}

\theoremstyle{remark}
\newtheorem{remark}[theorem]{Remark}

\numberwithin{equation}{section}

\newcommand{\ee}{\hskip0.15ex}

\newcommand{\dd}[1]{_{\raise-0.3ex\hbox{$\scriptstyle #1$}}}

\newcommand {\Norm}[2]{ \mathchoice 
    {|\ee #1\ee|\dd{#2}}
    {| #1 |_{#2}}
    {| #1 |_{#2}}
    {| #1 |_{#2}} }
\newcommand {\DNorm}[2]{ \mathchoice 
    {\|\ee #1\ee\|\dd{#2}}
    {\| #1 \|_{#2}}
    {\| #1 \|_{#2}}
    {\| #1 \|_{#2}} }
\newcommand {\Normc}[2]{ \mathchoice 
    {|\ee #1\ee|\dd{#2}^2}
    {| #1 |_{#2}^2}
    {| #1 |_{#2}^2}
    {| #1 |_{#2}^2} }
\newcommand {\DNormc}[2]{ \mathchoice 
    {\|\ee #1\ee\|\dd{#2}^2}
    {\| #1 \|_{#2}^2}
    {\| #1 \|_{#2}^2}
    {\| #1 \|_{#2}^2} }

\renewcommand{\div}{\operatorname{\rm div}}
\newcommand{\grad}{\operatorname{\textbf{grad}}}
\newcommand{\curl}{\operatorname{\textbf{curl}}}

\newcommand{\Cinf}{\mathscr{C}^\infty}

\newcommand\R{{\mathbb R}}

\renewcommand\L{{\mathscr L}}

\begin{document}

\title[Inequalities of Babu\v{s}ka-Aziz and Friedrichs-Velte for differential forms]{Inequalities of Babu\v{s}ka-Aziz\\ and Friedrichs-Velte for differential forms}  

\author{Martin Costabel}

\address{IRMAR UMR 6625 du CNRS, Universit\'{e} de Rennes 1\\
Campus de Beaulieu\\
35042 Rennes Cedex, France}

\email{martin.costabel@univ-rennes1.fr}

\date{19 June 2015}  



\begin{abstract}
For sufficiently smooth bounded plane domains, 
the equivalence between the inequalities of Babu\v{s}ka--Aziz for right inverses of the divergence and of Friedrichs on conjugate harmonic functions was shown by Horgan and Payne in 1983 \cite{HorganPayne1983}. In a previous paper \cite{CoDa_IBAFHP} we proved that this equivalence, and the equality between the associated constants, is true without any regularity condition on the domain. In three dimensions, Velte \cite{Velte98} studied a generalization of the notion of conjugate harmonic functions and corresponding generalizations of Friedrich's inequality, and he showed for sufficiently smooth simply-connected domains the equivalence with inf-sup conditions for the divergence and for the curl. For this equivalence, Zsupp\'an \cite{Zsuppan2013} observed that our proof can be adapted, proving the equality between the corresponding constants without regularity assumptions on the domain. 
Here we formulate a generalization of the Friedrichs inequality for conjugate harmonic differential forms on bounded open sets in any dimension that contains the situations studied by Horgan--Payne and Velte as special cases. We also formulate the corresponding inf-sup conditions or Babu\v{s}ka--Aziz inequalities and prove their equivalence with the Friedrichs inequalities, including equality between the corresponding constants. No a-priori conditions on the regularity of the open set nor on its topology are assumed.

\end{abstract}

\maketitle


\section{The inequalities}\label{S:Ineq}

\subsection{Notation}
Let $\Omega$ be a bounded open set in $\R^{n}$, $n\ge2$. We do not assume any regularity of the boundary nor any hypothesis on the topology of $\Omega$, not even connectedness.
By $\Lambda^{\ell}$, $0\le\ell\le n$, we denote the exterior algebra of $\R^{n}$, and we write $|\cdot|$ for the natural Euclidean norm on each space $\Lambda^{\ell}$. 
Differential forms of order $\ell$ with square integrable coefficients are denoted by 
$L^{2}(\Omega,\Lambda^{\ell})$, with norm
$$
  \|u\| = \Big( \int_{\Omega} |u(x)|^{2} \,dx \Big)^{\frac12} 
$$
and corresponding scalar product $\langle\cdot,\cdot\rangle$.
Similarly, we write
$\Cinf_{0}(\Omega,\Lambda^{\ell})$ for the space of smooth differential forms with compact support in $\Omega$. By
$$
  d: \Cinf_{0}(\Omega,\Lambda^{\ell}) \,\to\, \Cinf_{0}(\Omega,\Lambda^{\ell+1})
  \qquad\mbox{ and }\quad
  d^{*}: \Cinf_{0}(\Omega,\Lambda^{\ell}) \,\to\, \Cinf_{0}(\Omega,\Lambda^{\ell-1})
$$
we denote the exterior derivative $d$ and the coderivative $d^{*}$ which is the formal adjoint of $d$ with respect to the $L^{2}$ scalar product. For $\ell=n$, we set $d=0$ and for $\ell=0$, we set $d^{*}=0$. 

On $\Cinf_{0}(\Omega,\Lambda^{\ell})$, we consider the $H^{1}$ semi-norm defined by the (Hodge-)Laplacian:
$$
  \Norm{u}{1}= \langle\Delta u,u\rangle^{\frac12}\,.
$$ 
From the formula
$$
  \Delta = d^{*}d + dd^{*}
$$
we get
\begin{equation}
\label{E:H1semi}
 |u|_{1}^{2} = \|du\|^{2} + \|d^{*}u\|^{2}\;,
\end{equation}
valid for differential forms of any order.

In the following we will fix the order $\ell\in\{1,\dots,n-1\}$.

Let $H^{1}_{0}(\Omega,\Lambda^{\ell})$ be the completion of 
$\Cinf_{0}(\Omega,\Lambda^{\ell})$ under the $H^{1}$ semi-norm.
We then have the following continuous extensions of $d$ and $d^{*}$, still denoted by the same symbols:
\begin{equation}
\label{E:dinH10}
  d: H^{1}_{0}(\Omega,\Lambda^{\ell}) \,\to\, L^{2}(\Omega,\Lambda^{\ell+1})
  \qquad\mbox{ and }\quad
  d^{*}: L^{2}(\Omega,\Lambda^{\ell+1}) \,\to\,H^{-1}(\Omega,\Lambda^{\ell})
\end{equation} 
where $H^{-1}(\Omega,\Lambda^{\ell})$ is the dual space of $H^{1}_{0}(\Omega,\Lambda^{\ell})$.

In addition to the extensions \eqref{E:dinH10}, we shall also need the following extensions of $d$ and $d^{*}$:
\begin{equation}
\label{E:dinL2}
  d^{*}: H^{1}_{0}(\Omega,\Lambda^{\ell}) \,\to\, L^{2}(\Omega,\Lambda^{\ell-1})
  \qquad\mbox{ and }\quad
  d: L^{2}(\Omega,\Lambda^{\ell-1}) \,\to\,H^{-1}(\Omega,\Lambda^{\ell}).
\end{equation} 
If we need to disambiguate these extensions, we write $\underline d$ and $\underline d^{*}$ for the operators in \eqref{E:dinH10} and $\overline d$ and $\overline d^{*}$ for the operators in \eqref{E:dinL2}.

Finally, we define the closed subspace $M$ of $L^{2}(\Omega,\Lambda^{\ell-1})$ as the orthogonal complement of the kernel of $d$ in \eqref{E:dinL2}:
\begin{equation}
\label{E:defM}
  M = \{u\in L^{2}(\Omega,\Lambda^{\ell-1}) \mid 
  \forall\,v\in L^{2}(\Omega,\Lambda^{\ell-1}): dv=0 \Rightarrow \langle u,v\rangle = 0 \}\,.
\end{equation}
In the general situation, we immediately get the following result, which is a special case of the well-known fact that the kernel of a bounded operator between Hilbert spaces is the orthogonal complement of the range of the adjoint operator.
\begin{lemma}
\label{L:imd*denseinM}
The operator $d^{*}$ in \eqref{E:dinL2} maps  $H^{1}_{0}(\Omega,\Lambda^{\ell})$ to a dense subspace of $M$.
\end{lemma}
The question to be studied is whether the image of $d^{*}$ coincides with $M$. This is equivalent to the validity of the inf-sup condition or generalized Babu\v{s}ka--Aziz inequality that we will consider in the next subsection.

Let us note a consequence of the property $d^{*}\circ d^{*}=0$ of the de Rham complex and a corollary of Lemma~\ref{L:imd*denseinM}:
\begin{equation}
\label{E:Minkerd*}
 u\in M \;\Longrightarrow\; d^{*}u=0\,.
\end{equation}

\subsection{The Babu\v{s}ka--Aziz inequality}
The \emph{Babu\v{s}ka--Aziz constant} (of order $\ell$) $C_{\ell}(\Omega)$ of the open set $\Omega$ is defined as the smallest constant $C$ such that the following is true
\begin{equation}
\label{E:BA}
\text{
For any $q\in M$ there exists $v\in H^{1}_{0}(\Omega,\Lambda^{\ell})$ such that 
$d^{*}v=q$ and 
$|v|_{1}^{2}\le C \|q\|^{2}$.}
\end{equation}
If $d^{*}: H^{1}_{0}(\Omega,\Lambda^{\ell}) \,\to\,M$ is not surjective, we set 
$C_{\ell}(\Omega)=\infty$. Otherwise we say that the \emph{Babu\v{s}ka--Aziz inequality} holds.
Thus $C_{\ell}(\Omega)$ is finite if and only if $d^{*}$ is surjective onto $M$. By duality, this is the case if and only if $d$ is injective on $M$ and has closed range, that is, if and only if an a-priori estimate
\begin{equation}
\label{E:dLions}
 \forall\, q\in M\;:\;\|q\|^{2}\le C_{\ell}(\Omega) \, |dq|_{-1}^{2}
\end{equation}
is satisfied. Here $\Norm{\cdot}{-1}$ denotes the dual norm
$$
  \Norm{dq}{-1} = \sup_{v\in H^{1}_{0}(\Omega,\Lambda^{\ell})} 
     \frac{\langle dq,v\rangle}{\Norm{v}{1}} \,.
$$
Inserting this into \eqref{E:dLions}, we obtain the equivalent \emph{inf-sup condition}
\begin{equation}
\label{E:infsup}
 \inf_{q\in M} \sup_{v\in H^{1}_{0}(\Omega,\Lambda^{\ell})} 
     \frac{\langle q,d^{*}v\rangle}{\Norm{v}{1}} = \beta>0
     \quad\mbox{ with }\;
     \beta=C_{\ell}(\Omega)^{-\frac12} \,.
\end{equation}

\subsection{The Friedrichs--Velte inequality}
Generalizing the notion of conjugate harmonic functions, we consider pairs of differential $\ell-1$ forms $h\in L^{2}(\Omega,\Lambda^{\ell-1})$ and $\ell+1$ forms
$g\in L^{2}(\Omega,\Lambda^{\ell+1})$ satisfying 
\begin{equation}
\label{E:dh=d*g}
 dh = d^{*}g\,.
\end{equation}
Note that here $d=\overline d$ as defined in \eqref{E:dinL2} and $d^{*}=\underline d^{*}$ as defined in \eqref{E:dinH10}. 
We say that the \emph{Friedrichs--Velte inequality} holds if there is a constant $\Gamma$ such that
\begin{equation}
\label{E:FrV}
 \forall\,  h\in M, \, g\in L^{2}(\Omega,\Lambda^{\ell+1})
 \mbox{ satisfying \eqref{E:dh=d*g} : }\quad 
 \|h\|^{2} \le\,\Gamma\,\|g\|^{2}\,.
\end{equation}
The smallest possible constant $\Gamma$ will be denoted by $\Gamma_{\ell}(\Omega)$. 

\begin{remark}
Note that in the Friedrichs--Velte inequality \eqref{E:FrV} we assume that $h\in M$, which implies $d^{*}h=0$. In addition, we can assume without loss of generality that $g\in M^{*}$, where $M^{*}$ is defined in analogy to $M$ as the orthogonal complement in $L^{2}(\Omega,\Lambda^{\ell+1})$ of the kernel of $d^{*}$. The reason is that for fixed $h$, the element $g
\in L^{2}(\Omega,\Lambda^{\ell+1})$ with minimal $L^{2}$ norm that satisfies $d^{*}g=dh$ belongs to $M^{*}$. Thus $g$ can be assumed to satisfy $dg=0$. Both $h$ and $g$ are then harmonic, that is they belong to the kernel of the Hodge Laplacian $\Delta = d^{*}d + dd^{*}$.
The complete system for $h\in L^{2}(\Omega,\Lambda^{\ell-1})$ and $g\in L^{2}(\Omega,\Lambda^{\ell+1})$
$$
  d^{*}h=0\,,\quad dh = d^{*}g\,,\quad dg=0
$$
can be written in Clifford analysis notation 
$(d+d^{*})(g-h)=0$ and expresses the fact that the element $g-h$ is \emph{monogenic}. It is sometimes called generalized Moisil-Teodorescu system and has been studied as generalization of the definition of conjugate harmonic functions, see \cite{BrackxDelangheSommen2002}.
\end{remark}

\section{Equivalence between Babu\v{s}ka--Aziz and Friedrichs--Velte} \label{S:BA-FrV}

\begin{theorem}
\label{T:BA-FrV}
For any bounded open set $\Omega\subset\R^{n}$ and any $1\le\ell\le n-1$, the Babu\v{s}ka--Aziz constant $C_{\ell}(\Omega)$ is finite if and only if the Friedrichs--Velte constant $\Gamma_{\ell}(\Omega)$ is finite, and there holds
\begin{equation}
\label{E:CGamma}
 C_{\ell}(\Omega) = \Gamma_{\ell}(\Omega) + 1\,.
\end{equation}
\end{theorem}

\begin{proof}
The proof is divided into two parts.

\medskip\noindent
{\em (i)} In a first step, we assume that $\Omega$ is such that $C_{\ell}(\Omega)$ is finite. We will show that then $\Gamma_{\ell}(\Omega)$ is finite and
\begin{equation}
\label{E:Gamma<=C-1}
   \Gamma_{\ell}(\Omega) \le C_{\ell}(\Omega)-1.
\end{equation}

Let $h\in M$ and $g\in L^{2}(\Omega,\Lambda^{\ell+1})$ satisfy $dh=d^{*}g$.
From the Babu\v{s}ka--Aziz inequality we get the existence of $u\in H^{1}_{0}(\Omega,\Lambda^{\ell})$ such that 
\[
    d^{*}u= h\quad\mbox{ and }\quad 
   \|du\|^{2} = \Normc{u}{1} - \|d^{*}u\|^{2}
    \le (C_{\ell}(\Omega)-1)\,\|h\|^{2}\,.
\]
We find
\[
 \|h\|^{2}   = \langle h,d^{*}u\rangle 
               = \langle d h,u\rangle
               = \langle d^{*} g,u\rangle
               = \langle g,d u\rangle\,.
\]
With the Cauchy-Schwarz inequality and the estimate of $du$, we deduce
\[
   \|h\|^{2} \le \|g\| \,\|du\| \le \sqrt{C_{\ell}(\Omega)-1}\, \|g\| \,\|h\| \,,
\]
hence the estimate
\[
   \|h\|^{2}  \le (C_{\ell}(\Omega)-1) \, \|g\|^{2}\,,
\]
which proves \eqref{E:Gamma<=C-1}.

\medskip\noindent
{\em (ii)} In a second step, we assume that $\Omega$ is such that $\Gamma_{\ell}(\Omega)$ is finite. We will show that $C_{\ell}(\Omega)$ is finite and
\begin{equation}
\label{E:C<=Gamma+1}
   C_{\ell}(\Omega) \le \Gamma_{\ell}(\Omega)+1.
\end{equation}

Let $p\in M$ be given and define $u\in H^1_0(\Omega,\Lambda^{\ell})$ as the solution of
$\Delta u = dp$, that is $u$ is the unique solution of the variational problem
\begin{equation}
\label{E:u=w(p)}
  \forall v\in H^1_0(\Omega,\Lambda^{\ell}): \:
  \langle du,dv\rangle + \langle d^{*}u,d^{*}v\rangle = \langle p, d^{*}v\rangle
\,.
\end{equation}
We set $q=d^{*}u$ and $g=du$ and observe the following relations as consequences of \eqref{E:u=w(p)}:
\begin{alignat}{2}
\label{E:pqg}
 \langle p,q\rangle &= \Normc u1 &\,&=\DNormc q{} + \DNormc g{}\\
 \nonumber
 dp &= \Delta u && = d^{*}g +dq\\
\label{E:g2=(q,p-q)}
 \DNormc g{} &= \langle p,q\rangle -\DNormc q{} && = \langle p-q,q\rangle\,.
\end{alignat}
Noting that $q\in M$, we see that
 $h=p-q$ and $g$ are conjugate harmonic forms in the sense of \eqref{E:dh=d*g}. 
We can therefore use the Friedrichs--Velte inequality:
\begin{equation}
\label{E:Fri}
  \DNormc{p-q}{} \le \Gamma_{\ell}(\Omega)\, \DNormc g{}\,.
\end{equation}
Then we have with \eqref{E:g2=(q,p-q)} 
\[
  \DNormc g{} \le \DNorm q{} \DNorm{p-q}{} \le \DNorm q{} \sqrt{\Gamma_{\ell}(\Omega)}\DNorm g{}\,,
\]
hence
\begin{equation}
\label{E:Fri2}
 \DNormc g{} \le \Gamma_{\ell}(\Omega)\DNormc q{}\,.
\end{equation}
Now we estimate, using \eqref{E:pqg} and both \eqref{E:Fri} and \eqref{E:Fri2}:
\begin{align*}
 \DNormc p{} &= \DNormc{p-q}{} -\DNormc q{} + 2\langle p,q\rangle\\
   &= \DNormc{p-q}{} + \DNormc g{} + \DNormc q{} +\DNormc g{}\\
   &\le \Gamma_{\ell}(\Omega) \DNormc g{} + \Gamma_{\ell}(\Omega) \DNormc q{} + \DNormc q{} +\DNormc g{}\\
   &= \big(\Gamma_{\ell}(\Omega)+1\big)\, \Normc u{1} \,.
\end{align*}
We deduce 
$$
  \sup_{v\in H^{1}_{0}(\Omega,\Lambda^{\ell})} 
     \frac{\langle p,d^{*}v\rangle}{\Norm{v}{1}} \ge
  \frac{\langle p,d^{*}u\rangle }{\Norm u1}=\Norm u1 \ge 
    \frac1{\sqrt{\Gamma_{\ell}(\Omega)+1}}\, \|p\|\,.
$$
Therefore the inf-sup condition \eqref{E:infsup} is satisfied with 
$\beta\ge (\Gamma_{\ell}(\Omega)+1)^{-\frac12}$, and this gives the desired inequality \eqref{E:C<=Gamma+1}.

\smallskip
Theorem~\ref{T:BA-FrV} is proved. 
\end{proof}

The proof does not use any particular properties of the operators $d$ and $d^{*}$ except the formula \eqref{E:H1semi}. Therefore the Babu\v{s}ka--Aziz and Friedrich--Velte inequalities as well as the statement and proof of Theorem~\ref{T:BA-FrV} can be formulated in a more abstract setting that we will describe now. 
Comparing the situation of differential forms with the abstract setting will permit 
to clarify the role of the de Rham complex in this context. 

We need three Hilbert spaces $X$, $Y$, $Z$, where $Y$ and $Z$ are identified with their dual spaces, whereas $X$ is distinguished from its dual space $X'$, and two bounded linear operators
$$
  B:X\to Y\;,\qquad R:X\to Z.
$$
The object of interest is the operator $B$ and a possible a-priori estimate
\begin{equation}
\label{E:apB}
  \forall\, u\in \big(\ker B\big)^{\perp} \;:\; \|{u}\|_{X}^{2}\le C\, \|{Bu}\|_{Y}^{2}\,. 
\end{equation}
The validity of such an estimate is equivalent to the closedness of the range of $B$, and also to the analogous estimate for the dual operator $B^{*}$
$$
   \forall\, q\in \big(\ker B^{*}\big)^{\perp} \;:\; \|q\|_{Y}^{2}\le C\, \|{B^{*}q}\|_{X'}^{2}\,,
$$
as well as an inf-sup condition for $B$
$$
  \inf_{q\in (\ker B^{*})^{\perp\,}} \sup_{v\in X}
  \frac{\langle q, Bv\rangle}{\|q\|_{Y}\|v\|_{X}} = C^{-\frac12}>0\,.
$$
The smallest possible constant $C$ is the same in these three formulations and corresponds to the Babu\v{s}ka--Aziz constant.

The operator $R$ is an auxiliary operator that satisfies by assumption
\begin{equation}
\label{E:B*B+R*R}
 \forall\, v\in X\;:\; \|v\|_{X}^{2} = \|Bv\|_{Y}^{2} + \|Rv\|_{Z}^{2}\,.
\end{equation}
Such an operator $R$ will exist whenever $\|Bv\|_{Y}\le\|v\|_{X}$ holds in $X$.

We now call two elements $h\in Y$ and $g\in Z$ \emph{conjugate} if they satisfy
$B^{*}h=R^{*}g$, or equivalently
\begin{equation}
\label{E:B*h=R*g}
 \forall\, v\in X\;:\; \langle h,Bv\rangle = \langle g,Rv\rangle\,.
\end{equation}
The analog to the Friedrichs-Velte inequality is then:
\begin{equation}
\label{E:FVB}
 \mbox{ If $h\in Y$ and $g\in Z$ are conjugate and $h\in\big(\ker B^{*}\big)^{\perp}$, then }\; \|h\|_{Y}^{2} \le \Gamma \,\|g\|_{Z}^{2}\,. 
\end{equation}
Using the correspondence $B=d^{*}=\overline d^{*}$ defined in \eqref{E:dinL2} and $R=d=\underline d$ defined in \eqref{E:dinH10}, the preceding proof then immediately gives the following equivalence theorem.
\begin{theorem}
\label{T:BA=FrVforB}
The a-priori inequality \eqref{E:apB} is satisfied if and only if the inequality \eqref{E:FVB} between conjugate elements is satisfied, and for the smallest constants in the two inequalities there holds
$$
   C=\Gamma+1\,.
$$
\end{theorem}
The statement and proof of the abstract Theorem~\ref{T:BA=FrVforB} are verbatim translations of those of the result for the de Rham complex, Theorem~\ref{T:BA-FrV}, and we therefore do not repeat this proof.  
Where the concrete situation of operators $d$ and $d^{*}$ from the de Rham complex may provide further information is in the description of the space 
$M=(\ker B^{*}\big)^{\perp}$, see remarks at the end of Section~\ref{S:Ineq} and equation \eqref{E:MbyM0} in Section~\ref{S:Lip} below.


\section{Bounded Lipschitz domains}\label{S:Lip}

The classical Babu\v{s}ka--Aziz inequality in 2 and 3 dimensions  (which corresponds to $\ell=1$, see Section~\ref{S:Ex} below) has been known for a long time to be true for any bounded Lipschitz domain. Recently it has been shown to hold for the larger class of John domains \cite{AcostaDuranMuschietti2006}. While we do not yet know whether the generalization to other values of $\ell$ is also true for more general domains, we will show that it holds at least for Lipschitz domains. 

\begin{proposition}
\label{P:Lip}
Let the bounded open set $\Omega\subset\R^{n}$ be Lipschitz. Then for all $1\le\ell\le n-1$. the constants $C_{\ell}(\Omega)$ and $\Gamma_{\ell}(\Omega)$ are finite.
\end{proposition}

\begin{proof}
In view of the equivalence theorem~\ref{T:BA-FrV}, it is sufficient to prove the finiteness of the Babu\v{s}ka--Aziz constant $C_{\ell}(\Omega)$. This is equivalent to the surjectivity of $d^{*}$ from $H^{1}_{0}(\Omega,\Lambda^{\ell})$ to $M$, because the existence of the estimate in \eqref{E:BA} is then a consequence of the Banach open mapping theorem. In view of Lemma~\ref{L:imd*denseinM}, we only need to show that the image of $d^{*}$ is a closed subspace of $L^{2}(\Omega,\Lambda^{\ell-1})$, or equivalently, that $d$ from $L^{2}(\Omega,\Lambda^{\ell-1})$ to $H^{-1}(\Omega,\Lambda^{\ell})$ has closed range. This in turn is a consequence of the results of \cite{CostabelMcIntosh2010}, where it is shown that for any $s\in\R$, the range of
$$
 d: H^{s}(\Omega,\Lambda^{\ell-1}) \to H^{s-1}(\Omega,\Lambda^{\ell})
$$
is closed. For $s=0$, this is the desired result. 
\end{proof}

On a Lipschitz domain, the space $M$ can be characterized in more detail. 
The range $dH^{1}(\Omega,\Lambda^{\ell-2})$ has a finite-dimensional complement 
${\mathscr H}_{\ell-1}(\Omega)$
in the kernel of 
$$
 d:L^{2}(\Omega,\Lambda^{\ell-1})\to H^{-1}(\Omega,\Lambda^{\ell})\,,
$$ which can be chosen as the $L^{2}$-orthogonal complement. 
The orthogonal complement $M_{0}$ of $dH^{1}(\Omega,\Lambda^{\ell-2})$ in $L^{2}(\Omega,\Lambda^{\ell-1})$ is the kernel of 
$$
 d^{*}: L^{2}(\Omega,\Lambda^{\ell-1}) \to \big(H^{1}(\Omega,\Lambda^{\ell-2})\big)'
$$
which corresponds to the differential equation with normal boundary condition
$$
  M_{0} = \{v\in L^{2}(\Omega,\Lambda^{\ell-1}) \mid d^{*}v=0 \mbox{ in }\Omega;
    n\mathbin\lrcorner v=0\mbox{ on } \partial \Omega \}\,.
$$
Finally, it is not hard to see that the elements of $M$, in addition to satisfying this boundary value problem, are orthogonal to harmonic forms, so that we get the following equality.
\begin{equation}
\label{E:MbyM0}
  M = M_{0}\cap {\mathscr H}_{\ell-1}(\Omega)^{\perp}\,.
\end{equation}


\section{Examples}\label{S:Ex}

\subsection{The inf-sup constant for the divergence} 
This is the case $\ell=1$. 
We identify $0$-forms with scalars, $1$-forms with vectors of dimension $n$, and $2$-forms with $2$-vectors or antisymmetric matrices, of dimension $n(n-1)/2$. With this identification, the operators in \eqref{E:dinL2} are realized as
$$
  \overline d=\grad\,,\quad \overline d^{*}=-\div\,.
$$
In \eqref{E:dinH10}, the identification is
$$
  \underline d=\curl\,,\quad \underline d^{*}=\curl^{*}\,,
$$
where $\curl^{*}$ is the adjoint of the curl operator. 
In two dimensions, $\curl$ is the scalar curl and $\curl^{*}$ the vector curl, whereas in three dimensions, both $\curl$ and $\curl^{*}$ are represented by the standard curl operator.

The space $M$ is the orthogonal complement in $L^{2}(\Omega)$ of the kernel of the gradient, which consists of locally constant functions. Thus $M$ is the space of functions of vanishing integral over each connected component of $\Omega$; in the case of a connected open set $\Omega$, it is the space $L^{2}_{\circ}(\Omega)$ of functions of mean value zero.

This is the classical situation of the Babu\v{s}ka--Aziz inequality and of the inf-sup condition
$$
  \inf_{q\in L^{2}_{\circ}(\Omega)} \sup_{v\in H^{1}_{0}(\Omega)^{n}}
  \frac{\langle q,\div v\rangle}{\|q\| \Norm{v}{1}} = \beta>0, \;
  \beta=C_{1}(\Omega)^{-\frac12}\,,
$$
 the validity of which for bounded Lipschitz domains in any dimension has been known since the work of Bogovski\u{\i} \cite{Bogovskii79}.
 
In this situation, the Friedrichs-Velte inequality \eqref{E:FrV} corresponds to the following:

\textbf{In two dimensions}, \eqref{E:dh=d*g} is the Cauchy-Riemann system and \eqref{E:FrV} can be written as:
$$
 \forall\,  h\in L^{2}_{\circ}(\Omega), \, g\in L^{2}(\Omega)
 \mbox{ such that $h+ig$ is holomorphic : }\quad 
 \|h\|^{2} \le\,\Gamma_{1}(\Omega)\,\|g\|^{2}\,.
$$
This is the inequality studied by Friedrichs \cite{Friedrichs1937}. Its equivalence with the Babu\v{s}ka--Aziz inequality was shown for $C^{2}$ domains by Horgan--Payne \cite{HorganPayne1983} and without regularity assumptions in \cite{CoDa_IBAFHP}. Horgan--Payne applied this equivalence together with an estimate for the Friedrichs constant to get an upper bound for the Babu\v{s}ka--Aziz constant for star-shaped domains in two dimensions.

\textbf{In three dimensions}, \eqref{E:FrV} can be written as:
$$
 \forall\,  h\in L^{2}_{\circ}(\Omega), \, \boldsymbol{g}\in L^{2}(\Omega)^{3}
 \mbox{ such that $\grad h = \curl \boldsymbol{g}$ : }\quad 
 \|h\|^{2} \le\,\Gamma_{1}(\Omega)\,\|\boldsymbol{g}\|^{2}\,.
$$
This inequality was shown by Velte \cite{Velte98} and its equivalence with the Babu\v{s}ka--Aziz inequality was proved for $C^{2}$ domains. Zsupp\'an \cite{Zsuppan2013} found that our proof for the two-dimensional case in \cite{CoDa_IBAFHP} can be generalized to prove the equivalence without regularity assumptions. In \cite{Payne2007}, Payne proved an estimate of the Velte constant $\Gamma_{1}(\Omega)$ for star-shaped domains and used Velte's equivalence result to get an upper bound for the Babu\v{s}ka--Aziz constant $C_{1}(\Omega)$ in three dimensions.

\subsection{The inf-sup constant for the curl} \label{infsupcurl}
Let us consider the case $n=3, \ell=2$. In this case, $\ell-1$ forms and $\ell$ forms are represented by vector functions and $\ell+1$ forms are scalar. 
The operators in \eqref{E:dinL2} are realized as
$$
  \overline d=\curl\,,\quad \overline d^{*}=\curl\,.
$$
In \eqref{E:dinH10}, the identification is
$$
  \underline d=\div\,,\quad \underline d^{*}=-\grad\,,
$$
The space $M\subset L^{2}(\Omega)^{3}$ is now the orthogonal complement of the kernel of $\curl$. If the domain $\Omega$ is \emph{simply connected}, then this kernel consists of the gradients of $H^{1}(\Omega)$ functions. Orthogonality to these gradients is the variational description of the space
$$
  M= H_{0}(\div0,\Omega)=\{v\in L^{2}(\Omega)^{3} \mid
    \div v=0 \mbox{ in }\Omega; \;
    {n}\cdot v=0 \mbox{ on }\partial\Omega\}\,.
$$ 
This is the setting in which Velte \cite{Velte98} formulated his second version of the Friedrichs inequality
$$
  \forall\, h\in H_{0}(\div0,\Omega), g\in L^{2}(\Omega)^{3}:
  \curl h = \grad g \Longrightarrow \|h\|^{2} \le \Gamma_{2}(\Omega) \|g\|^{2}
$$
 and proved its equivalence with the a-priori estimates for the curl in $H^{1}_{0}(\Omega)^{3}$
$$
  \forall\, v\in (\ker \curl)^{\perp} \;:\; \Norm{v}1^{2}\le C_{2}(\Omega)\,\DNorm{\curl v}{}^{2}\;\;
$$
and in $L^{2}(\Omega)^{3}$
$$
  \forall\, v\in H_{0}(\div0,\Omega) \;:\; \DNorm{v}{}^{2}\le C_{2}(\Omega)\,|\curl v|_{-1}^{2}\,.
$$
This is equivalent to the inf-sup condition
$$
  \inf_{q\in H_{0}(\div0,\Omega)} \sup_{v\in H^{1}_{0}(\Omega)^{3}}
  \frac{\langle q,\curl v\rangle}{\|q\| \Norm{v}1} = C_{2}(\Omega)^{-\frac12}\,.
$$
If $\Omega$ is not simply connected, then there exists a finite-dimensional cohomology space 
$${\mathscr H}_{1}(\Omega)=H_{0}(\div0,\Omega)\cap \ker\curl$$
so that in the above inequalities $H_{0}(\div0,\Omega)$ has to be replaced by the smaller space
$$
  M= H_{0}(\div0,\Omega) \cap {\mathscr H}_{1}(\Omega)^{\perp}\,.
$$



\def\cprime{$'$}

\end{document}